\documentclass{amsart}

\usepackage{amssymb}
\usepackage{amsmath} 
\usepackage{amsthm} 
\usepackage{bm}
\usepackage{calc} 
\usepackage{enumerate}
\usepackage{graphicx,epsfig,color,url}


\date{\today}
\def\cdiv{\mathop{\mathrm{crit}}}
\def\J{{\mathcal{J}}}
\def\aut{{\mathop{\mathrm{Aut}}}}
\def\A{{\mathcal A}}
\def\N{{\mathbb N}}
\def\R{{\mathbb R}}
\def\C{{\mathbb C}}
\def\E{{\mathcal E}}

\def\I{{\mathcal I}}
\def\K{{\mathcal K}}
\def\L{{\mathcal L}}
\def\P{{\mathbb P}}
\def\CP{{\mathbb{CP}}}
\def\RP{{\mathbb{RP}}}
\def\U{{\mathcal U}}
\def\deg{{\rm deg}}
\def\supp{\mathop{\rm supp}}
\def\T{{\mathcal T}}
\def\fp{{e}}

\renewcommand{\hat}{\widehat}

\def\tto{\dashrightarrow}

\def\deg{\mathop{\rm deg}}
\def\dist{{\rm dist}}
\renewcommand\div{\mathop{\mathrm{Div}}}

\newtheorem{theorem}{\bf Theorem}[section]
\newtheorem{proposition}[theorem]{\bf Proposition}

\newtheorem*{theorem*}{\bf Theorem}
\newtheorem{lemma}[theorem]{\bf Lemma}
\newtheorem{corollary}[theorem]{\bf Corollary}

\newtheorem{remark}[theorem]{\bf Remark}
\newtheorem*{remark*}{\bf Remark}

\theoremstyle{plain}
\newtheorem*{thma}{Theorem A}
\newcommand{\mainthma}{{A}}
\newtheorem*{thmb}{Theorem B}
\newcommand{\mainthmb}{{B}}

\newcommand{\p}{{p_0}}

\begin{document}

\title[Typical dynamics of plane rational maps with equal degrees]
{Typical dynamics of plane rational maps with \\ equal degrees}

\begin{author}[J. Diller]{Jeffrey Diller}
\email{diller.1@nd.edu}
\address{%
Department of Mathematics\\
University of Notre Dame\\
Notre Dame IN 46556\\
United States
}
\end{author}

\begin{author}[H. Liu]{Han Liu}
\email{hliu@qq.com}
\address{%
Department of Mathematics\\
University of Notre Dame\\
Notre Dame IN 46556\\
United States
}
\end{author}

\begin{author}[R. K. W. Roeder]{Roland K. W. Roeder}
\email{rroeder@math.iupui.edu}
\address{ %
IUPUI Department of Mathematical Sciences\\
LD Building, Room 224Q\\
402 North Blackford Street\\
Indianapolis, Indiana 46202\\
 United States }
\end{author}

\date{\today}

\begin{abstract}
Let $f:\CP^2\tto\CP^2$ be a rational map with algebraic and topological degrees
both equal to $d\geq 2$.  Little is known in general about the ergodic
properties of such maps.  We show here, however, that for an open set of
automorphisms $T:\CP^2\to\CP^2$, the perturbed map $T\circ f$ admits  exactly
two ergodic measures of maximal entropy $\log d$, one of saddle and one of
repelling type.  Neither measure is supported in an algebraic curve,
and $f_T$ is `fully two dimensional' in the sense that it does not preserve any
singular holomorphic foliation of $\C\P^2$.  In fact, absence of an invariant foliation
extends to all $T$ outside a countable union of algebraic subsets of $\aut(\P^2)$.  Finally, we illustrate all of our
results in a more concrete particular instance connected with a two dimensional
version of the well-known quadratic Chebyshev map.
\end{abstract}

\maketitle

\section*{Introduction}
Let $f: \P^2 \tto \P^2$ be a dominant rational self-map of the complex
projective plane~$\P^2$.  A great deal of effort has gone into understanding how the
ergodic theory of $f$ is governed by two numerical invariants, the first and
second dynamical degrees $\lambda_1(f)$ and $\lambda_2(f)$, of $f$.  The two
cases $\lambda_2(f) > \lambda_1(f)$ and $\lambda_1(f) > \lambda_2(f)$ are well-studied, the former corresponding to ``predominantly repelling dynamics'' and the latter to ``predominantly saddle-type
dynamics.''  The borderline case $\lambda_1 = \lambda_2$ has received much less attention.
In this paper we provide many examples of rational maps $f: \P^2 \tto
\P^2$ with $\lambda_1(f) = \lambda_2(f)$ that exhibit both repelling and
saddle-type dynamics equally.

In order to situate and state our main results, let us discuss what is and is not known in a bit more detail.
If we write $f$ in homogeneous coordinates as $f = [f_1:f_2:f_3]$, with the $f_i$ having no common factors of positive degree, then the common degree of the $f_i$ is called the {\em algebraic degree} $d(f)$.  The {\em first dynamical degree} is then the asymptotic growth rate
\begin{equation}
\label{eqn:firstdegree}
\lambda_1(f) := \lim_{n \rightarrow \infty} d(f^n)^{1/n}
\end{equation}
of the algebraic degrees.  If $d(f^n) = d(f)^n$ for all $n$ then $f$ is said to be
{\em algebraically stable} and we have $\lambda_1(f) = d(f)$.

The {\em topological degree} $\lambda_2(f)$ is the number of preimages $\# f^{-1}(p)$ of a generic point $p \in \P^2$.  It is well-behaved under
iteration, satisfying $\lambda_2(f^n) = \lambda_2(f)^n$, and thus is also called the {\em second dynamical degree} of $f$.

The basic ergodic theoretic invariant of $f$ is its topological entropy $h_{top}(f)$.  Gromov \cite{gro} and Dinh-Sibony \cite{DS_ENTROPY} showed that entropy is controlled by the dynamical degrees.  Specifically,
\begin{equation}
\label{eqn:gromovbd}
h_{top}(f) \leq \log\max\{\lambda_1(f),\lambda_2(f)\}.
\end{equation}
Guedj \cite{Guedj} and Dinh-Nguyen-Truong \cite{DNT}, following Briend and Duval
\cite{BD2}, showed further that when $\lambda_2>\lambda_1$, the map $f$ has a
unique invariant measure $\mu$ of maximal entropy $\log \lambda_2$, that $\mu$
is mixing and repelling (both Lyapunov exponents are positive), that repelling
periodic points for $f$ are asymptotically equidistributed with respect to
$\mu$, and that $\mu$ does not charge small (i.e. pluripolar) sets.

On the other hand, work of Diller-Dujardin-Guedj \cite{ddg2,ddg3} gives that when $\lambda_1~>~\lambda_2$ and certain
additional technical hypotheses are satisfied, then $f$ has (again) a mixing invariant measure $\nu$ of maximal entropy $\log \lambda_1$.  But in this case $\nu$ is of \emph{saddle} type (one Lyapunov exponent is positive and the other is negative), and saddle periodic points for $f$ are asymptotically equidistributed with respect to~$\nu$.  Uniqueness of the measure of maximal entropy when $\lambda_1>\lambda_2$ has been established only in special cases, notably for polynomial automorphisms of $\C^2$ \cite{BLS} and for surface automorphisms \cite{CANTAT_K3}.

Almost nothing beyond the bound on topological entropy is known about the ergodic theory of rational maps with equal dynamical degrees $\lambda_1 = \lambda_2$.  Products $f:=(f_1,f_2)$ of one dimensional maps $f_j:\P^1\to\P^1$ in which one of the factors is linear furnish simple examples of plane rational maps with $\lambda_1=\lambda_2$, and these suggest some possibilities for the ergodic theory.
Consider for example the rational maps $f,g,h: \P^2 \tto \P^2$, given on $\C^2$ by
\begin{align*}
f(x,y) = (x+1,y^2), \,\, g(x,y) = (2x,y^2), \,\, \mbox{and} \,\, h(x,y) = (x,y^2).
\end{align*}
Each of these has $\lambda_1 = \lambda_2 = 2$.  The map $f: \P^2
\tto \P^2$ has topological entropy~$0$; see  \cite[Example 1.4]{GUEDJ_ETDS}.
Meanwhile, $g : \P^2 \tto \P^2$ has a unique measure $\mu$ of maximal entropy
$\log 2$ which is normalized Lebesgue measure on the unit circle
$\{x=0\} \times \{|y| = 1\}$.  However, if one changes surface, compactifying $\C^2$ by $\P^1\times \P^1$ instead of $\P^2$, the resulting map $g: \P^1\times \P^1 \rightarrow \P^1 \times \P^1$ acquires a second measure of maximal entropy $\nu$ given by
normalized Lebesgue measure on $\{x = \infty\} \times \{|y| = 1\}$.  The measure $\mu$ is repelling, while $\nu$ is of saddle type.
Finally, for each $x_0 \in \C$ the map $h$ has normalized Lebesgue measure on the unit circle $\{|y| = 1\}$
as measure of maximal entropy within each vertical complex line
$\{x=x_0\}$.

Of course, not all maps with equal dynamical degrees are products, but many
non-product examples $f:\P^2\tto\P^2$ are still semiconjugate to rational maps 
$\check f~:~\P^1~\to~\P^1$ via a some rational map $\P^2\tto\P^1$.  Such maps are
said to have `invariant fibrations'.  They include \cite{df} all maps with
$\lambda_1=\lambda_2 = 1$ (i.e. all examples where $f$ is birational) and
examples \cite{BG,BGN,GZ,SABOT} that arise in connection with spectral theory
for operators on self-similar spaces.  Based on this evidence, Guedj asked
whether any rational map with equal (maximal) dynamical degrees must preserve a
fibration \cite[p.103]{GUEDJ_HABILITATION}.   However, examples were recently
found by Bedford-Cantat-Kim \cite{BCK} and Kaschner-P\'erez-Roeder \cite{KPR}
showing that this is not the case.  

The product examples $f,g,h$ defined above are each degenerate in another way.  The ergodic measures of maximal entropy for these three maps are all supported in algebraic curves (lines, in fact).  Since birational changes of coordinate can contract curves, such measures are not very robust.  The measures obtained in \cite{Guedj} and \cite{ddg2} for `cohomologically hyperbolic' cases $\lambda_1\neq \lambda_2$ are much more diffuse; in particular they do not charge algebraic curves.

The main results of this paper show that there are large families of rational maps with equal dynamical degrees $\lambda_1 = \lambda_2$ that admit exactly two measures of maximal entropy without any of the above problems.

\begin{thma} 
Let $f:\P^2\tto\P^2$ be a rational map with equal algebraic and topological degrees $d(f) = \lambda_2(f)\geq 2$.  Then there is an open subset of linear maps $T\in \aut(\P^2)$ for which the map $f_T := T\circ f$ satisfies the following.
\begin{itemize}
 \item $f_T$ is algebraically stable and therefore has  $d(f) = \lambda_1(f_T) = \lambda_2(f_T)$.
 \item There is no $f_T$ invariant foliation; in particular, $f_T$ is not rationally semiconjugate to a holomorphic self-map of a Riemann surface.
 \item There are exactly two $f_T$-invariant and ergodic measures $\mu$ and $\nu$ of maximal entropy $\log d(f)$.
 \item $f_T$ is uniformly expanding on $\supp\mu$.
 \item $f_T$ is uniformly hyperbolic of saddle type on $\supp\nu$.
 \item Neither $\supp\nu$ nor $\supp\mu$ are contained in any algebraic curve.
 \item Any point $p\in \P^2$ whose forward images are not indeterminate for $f_T$ has a forward orbit asymptotic to either $\supp\mu$, $\supp\nu$ or one of finitely many attracting periodic points.  
\end{itemize}
\end{thma}

\begin{remark*}
It is easy to find mappings $f$ satisfying the hypothesis of Theorem~A.  For
example, suppose $g: \P^2 \rightarrow \P^2$ is an endomorphism and $h: \P^2
\dashrightarrow \P^2$ is a birational map, both of algebraic degree $a$.  Then,
$f: = g\circ h$ has $d(f)~=~\lambda_2(f)~=~a^2$.
\end{remark*}

\vspace{-0.01in}
The subset of $\aut(\P^2)$ from Theorem A is open with respect to the metric topology on $\aut(\P^2)$, not the Zariski topology.   (In fact, we will never use the Zariski topology in this paper.)
Therefore, we do not know whether there might be different, but also fairly robust possibilities for the ergodic theory of a rational map with equal degrees.  Are there for instance, large families of such maps with exactly one measure of maximal entropy?  More than two?  However each of the first two conclusions of Theorem {\mainthma } are obtained by (essentially) verifying that $T$ can be chosen to avoid countably many algebraic coincidences.  So once the conclusions hold for one $T$, they apply `generically.'  This is the content of the next result, which answers a question posed to us by Charles Favre.
\vspace{-0.01in}

\begin{thmb}
Let $f:\P^2\tto\P^2$ be a rational map with equal algebraic and topological degrees $d(f) = \lambda_2(f)\geq 2$.  Then for all $T\in\aut(\P^2)$ outside a countable union of proper algebraic subsets, 
\begin{itemize}
 \item $f_T$ is algebraically stable, with $\lambda_1(f_T) = \lambda_2(f_T)$.
 \item There is no $f_T$ invariant foliation.
\end{itemize}
\end{thmb}
\noindent

The connection between dynamical degrees and ergodic theory
extends to the much more general context of meromorphic self-maps on
compact K\"ahler manifolds of any dimension.  The reader may consult~\cite{GUEDJ_ETDS} to gain a good sense of the larger picture, much of which is still conjectural.  However, a recent paper
of Vigny~\cite{VIGNY} validates this picture for ``generic cohomologically
hyperbolic rational maps'' of $\P^k$.

\vspace{0.05in}

Some necessary background on the dynamics of rational maps is presented in~\S \ref{SEC_REF}. \S \ref{SEC_PROOFS} gives the proofs of Theorems {\mainthma } and \mainthmb.  The main techniques that we use are hyperbolic dynamics, elementary geometry of algebraic curves, and an application of Theorem 4.1' from \cite{KPR}.  One can obtain some (though not nearly all) of the conclusions of Theorem {\mainthma } with less exertion by appealing to work of De Th{\'e}lin and Vigny \cite[Theorem 1]{DE_THELIN_VIGNY}.  In order to keep the discussion more self-contained, we do not take this route here.

Given a rational map $f:\P^2\tto\P^2$ with equal degrees $d(f) = \lambda_2(f)$,
one can in principle use our arguments for Theorem {\mainthma } to identify a
specific linear transformation $T\in\aut(\P^2)$ for which $f_T$ satisfies all
conclusions of the Theorem.  In practice however, it seems a bit daunting to
arrange everything we need from $T$ to make the theorem work.  Hence in the
final section \S \ref{SEC_PERTURBED_CHEBYSHEV} of this paper, we consider a
family of rational maps $f_t:\P^2\tto\P^2$ that depend on a single real
parameter $t\in(0,1]$ and for which symmetry considerations allow us to verify
the conclusions of Theorem {\mainthma } more directly.  The initial map $f =
f_1$ is closely related to the one variable quadratic Chebyshev map.  We show
that the conclusions of Theorem {\mainthma } hold for all $t\in (0,1]$ close
enough to $0$.  An interesting additional outcome is that for all such~$t$, the
measure $\mu$ is real, supported on $\R\P^2\subset\P^2$, whereas the measure
$\nu$ is not.

\section{Rational maps, dynamical degrees and entropy}\label{SEC_REF}
Throughout this paper $\P^2$ will denote the complex projective plane.  Unless otherwise specified, we will measure distance between points with the usual Fubini-Study metric, normalized so that $\P^2$ has unit volume.  For any set $X\subset \P^2$ and $r>0$, we let $B_r(X) := \{p\in\P^2:\dist(p,X)<r\}$ be the $r$-neighborhood of $X$. 

Henceforth $f:\P^2\tto\P^2$ will denote a rational map of the complex
projective plane $\P^2$.  Let us recall some definitions, facts and
conventions concerning such maps.  In homogeneous coordinates, $f$ is given by
$f = [f_1:f_2:f_3]$, where the components $f_j[x_1:x_2:x_3]$ are homogeneous
polynomials, all of the same degree $d$ and with no non-constant common
factors.  The common zeroes of the components $f_j$ correspond to a finite set 
$\I = \I(f)\subset\P^2$ on which $f$ is not definable as a continuous map.  On $\P^2 \setminus \I$, the map
$f$ is well-defined and holomorphic.   We define the ``image" of a point $p\in
\I$ to be the set $f(p)$ of all possible limits $\lim_{p_j\to p} f(p_j)$.  This
is always a non-trivial algebraic curve.  Note that here and in the rest of the paper, algebraic curves are allowed to be reducible unless otherwise noted.

We assume throughout that $f$ is {\em dominant}, meaning that $f(\P^2 \setminus \I)$ is not contained in an algebraic curve.
We define the `image' of an algebraic curve $V\subset\P^2$ under $f$ to be its (set-theoretic) proper transform
\begin{align}\label{EQN_IMAGE_CURVE}
f(V) := \overline{f(V\setminus \I)}.
\end{align}
If $V$ is irreducible, then $f(V)$ is either another irreducible algebraic
curve or a point.   For a rational map of $\P^2$,
the latter can only happen if $V \cap \I \neq \emptyset$; see \cite[Proposition 1.2]{FS2}.

If $V$ is an irreducible curve with $f(V)$ a point, we call $V$
\emph{exceptional} for $f$, letting $\E$ denote the union of all (finitely
many) exceptional curves.  
We define the preimage $f^{-1}(V)$ of a
curve $V$ to be the union of all non-exceptional irreducible curves $C$ such
that $f(C) \subset V$.  A curve $V\subset X$ is \emph{invariant} if $f(V) = V$
and \emph{totally invariant} if, $f^{-1}(V) = V$.  

We say that an irreducible  curve $V$ is \emph{ramified} for $f$ if $V$ is not exceptional, but $f$ is not locally one-to-one near any point in $V$.  The collection of all exceptional and ramified curves forms the critical set $\cdiv(f)$ of $f$.  These may be assembled with multiplicities defined by the order of vanishing of the Jacobian determinant into the \emph{critical divisor} $\cdiv(f) \in \div(\P^2)$ of $f$.  

As in the introduction, we let $d=d(f)$, $\lambda_1 = \lambda_1(f)$ and $\lambda_2 = \lambda_2(f)$ denote the algebraic and first and second dynamical degrees of $f$.  The integer $\lambda_2$ is (alternatively) the multiplier for the induced action $f^*$ on $H^4(\P^2,\R) \cong \R$.  Likewise, the algebraic degree $d$ of $f$ is the multiplier of the pullback action $f^*$ on $H^2(\P^2,\R)\cong \R$.   Indeed any divisor $D\in\div(\P^2)$ admits (see e.g. \cite[\S 1.2]{ddg1}) a natural pullback $f^*D$ and pushforward $f_*D$ by $f$, and the resulting divisors satisfy $\deg f_*D = \deg f^*D = d\cdot \deg D$. 

We point out that if $C$ is an irreducible curve, then the support of $f_*C$ is equal to the \emph{total transform} of $C$ by $f$, i.e. to the proper transform $f(C)$ together with the images of all points in $\I(f)\cap C$.  In particular $\supp f_*C$ will typically be reducible when $C$ meets $\I(f)$.  On the other hand, when $C$ is disjoint from $\I(f)$, we have $\supp f_*C = f(C)$ is irreducible and more precisely, $f_*C = \lambda f(C)$ where $\lambda$ is the topological degree of the restriction $f|_C:C\to f(C)$. 

Likewise, if $C$ is irreducible and does not contain the image of any exceptional curve, then 
$$
f^*C = \sum_{f(C') = C} \nu' C',
$$
where the sum is over irreducible curves $C'$, and $f$ is locally $\nu'$-to-1 near a general point in $C'$.

One always has that $\lambda_2(f^n) = \lambda_2(f)^n$, and typically also $d(f^n) = d^n$.  However, when $\I\neq \emptyset$ the latter equality can fail.  There is a useful geometric characterization of (the opposite of) this possibility.

\begin{proposition}[See \cite{FS2, df}]
\label{prop:ascriterion}
The following are equivalent.
\begin{itemize}
 \item $d(f^n) = d^n$ for all $n\in\N$;
 \item $\lambda_1(f) = d$;
 \item No exceptional curve has forward orbit containing a point of indeterminacy. 
\end{itemize}
\end{proposition}

\noindent As in the introduction, we will call $f$ \emph{algebraically stable} if one/all of the conditions in this proposition hold.  

The upper bound \eqref{eqn:gromovbd} for the topological entropy\footnote{We remark that there is some subtlety defining $h_{top}(f)$ if $\mathcal{I} \neq \emptyset$.  This issue is addressed in~\cite{DS_ENTROPY}.} shows the dynamical significance of dynamical degrees.  
It is expected that this bound is actually an equality in most situations, and more precisely, that there exists an $f$-invariant measure $\mu$ whose metric entropy satisfies $h_{\mu}(f) = \log\max\{\lambda_1,\lambda_2\}$.  Theorem {\mainthma } of this article concerns the nature of such measures when $\lambda_1=\lambda_2$.  

\section{Proof of Theorems {\mainthma } and \mainthmb}\label{SEC_PROOFS}

\subsection{Linear perturbations of plane rational maps}

For any linear map $T\in\aut(\P^2)$, we will let $f_T$ denote the `perturbed' map $T\circ f$.  Note that the algebraic and topological degrees of $f_T$ are the same as those of $f$.  So are the indeterminacy, exceptional and critical sets of $f_T$.

We will be interested in choosing $T$ so that most points of $\P^2$ map by $f_T$ into a small neighborhood of some fixed line.  More specifically, for the remainder of this section we choose a line $L_0\subset\P^2$, a point $\p\in \P^2$ not contained in $L_0$ and a surjective linear map $T_0:\P^2\setminus\{\p\}\to L_0$.  As the discussion proceeds we will need to impose further conditions on these choices.  Let us list all of these now.
\begin{itemize}
 \item[(A)] $\p\notin f(\I)\cup f(L_0)$;
 \item[(B)] $L_0\cap \I = \emptyset$;
 \item[(C)] $\p$ is a regular value of $f$, with $\lambda_2$ distinct preimages; 
 \item[(D)] $\deg f(L_0) \geq 2$.  In particular, $L_0$ is not exceptional.
\end{itemize}

Property (B) implies that we can choose 
$\epsilon > 0$ sufficiently small so that 
\begin{align}\label{EQN_NO_INDET_IN_TUBE}
B_\epsilon(L_0) \cap \I = \emptyset.
\end{align}
Since any exceptional curve for $f$ must intersect $\I$,
this also implies no exceptional curve is contained in $B_\epsilon(L_0)$.  For the remainder of this section, we will suppose
$\epsilon > 0$ is sufficiently small that (\ref{EQN_NO_INDET_IN_TUBE}) holds.

Note that for any $\delta>0$ the set
$$
\T(T_0,\delta) := \{T\in\aut(\P^2): \dist(T(p),T_0(p))< \delta \text{ for all } p\notin B_\delta(\p)\}
$$
is non-empty and open in $\aut(\P^2)$.  

\begin{proposition}
\label{prop:twotraps}
Suppose that conditions (A) and (B) hold and $\epsilon > 0$ is small enough that $B_\epsilon(L_0)$ and $B_\epsilon(f^{-1}(p_0))$ are disjoint from each other and from $\I$.  Then there exists $\delta >0$ such that for all $T\in \T(T_0,\delta)$
\begin{itemize}
 \item $f_T(\P^2\setminus B_\epsilon(f^{-1}(p_0))) \subset B_\epsilon(L_0)$;
 \item $f_T^{-1}(\P^2\setminus B_\epsilon(L_0)) \subset B_\epsilon(f^{-1}(p_0))$.
\end{itemize}
In particular, $f_T(\I)\subset B_\epsilon(L_0)$ and any ergodic $f_T$-invariant measure is supported entirely in $B_\epsilon(L_0)$ or in $B_\epsilon(f_T^{-1}(p_0))$.  

If in addition, condition (C) holds, then we may further arrange that all points in  $B_\epsilon(f^{-1}(p_0))$ are regular for $f_T$ and (therefore) $B_\epsilon(L_0)$ contains all critical values of $f_T$.
\end{proposition}

\begin{proof} Straightforward check.
\end{proof}

\begin{proposition}\label{PROP_AS}
If conditions (A) and (B) hold and $\delta>0$ is small enough, then for all $T\in\T(T_0,\delta)$, the map $f_T:\P^2\tto\P^2$ is algebraically stable.  
\end{proposition}

\begin{proof}
Choose $\epsilon$ and then $\delta$ as in Proposition \ref{prop:twotraps}.  Since each exceptional curve $C$ for $f$ meets $L_0\subset B_\epsilon(L_0)$, we have $f_T^n(C)\in B_\epsilon(L_0)$ for all $n\geq 1$.  In particular, the forward orbit of $C$ is disjoint from $\I$, i.e. $f_T$ is algebraically stable.
\end{proof}

\begin{corollary}\label{COR_AS}
The map $f_T$ is algebraically stable for all $T\in\aut(\P^2)$ outside a countable union of proper algebraic subsets.
\end{corollary}

\begin{proof}
Given $n\in\N$ and a point $p\in \P^2$, the condition $p\in\I(f_T^n)$ amounts to an algebraic constraint on $T$.  By Proposition \ref{prop:ascriterion}, $f_T$ fails to be algebraically stable precisely when there exists some (smallest) $n\in\N$ and a point $p$ in the finite set $f_T(\E)$ such that $p\in\I(f_T^n)$.  Accounting for all possible $n\in\N$, we see that $f_T$ is algebraically stable for all $T$ outside a countable union of algebraic subsets of $\aut(\P^2)$.  Proposition \ref{PROP_AS} guarantees that these subsets are all proper.
\end{proof}

\subsection{(No) invariant foliations for maps with equal degrees}

The map $f:\P^2~\tto~\P^2$ is said to `preserve a rational fibration' if it is
rationally semiconjugate to a one dimensional map, i.e. if there is a rational
map $\phi:\P^2\tto \P^1$ and a rational function $\check f:\P^1 \to \P^1$ such
that $\phi\circ f = \check f\circ \phi$ (at all points where both compositions
are defined).  If either the base map $\check f$, or the fiber map
$f_x:\phi^{-1}(x) \to \phi^{-1}(f(x))$ ($x\in\P^1$ a general point) has degree
one, then we have (see e.g. Lemma 4.1 in \cite{ddg1}) equality $\lambda_1(f) =
\lambda_2(f)$ of the dynamical degrees of $f$. 

It was demonstrated recently \cite{KPR} that $\lambda_1(f) =
\lambda_2(f)$ does not imply $f$ preserves a rational fibration, nor indeed even a singular holomorphic foliation.
Specifically, the following criterion appears in that paper:

\begin{theorem}[Theorem 4.1' from \cite{KPR}]\label{THM_CONDITIONS_FOR_NO_FOLIATION}
Assume that $f$ has an indeterminate point $q$ such that
  \begin{enumerate}
    \item there are irreducible curves $C_n\subset f^n(q)$ with $\limsup \, \deg(C_n) = \infty$; and
    \item $q$ has an infinite preorbit along which $f$ is a finite holomorphic map.
  \end{enumerate}
  Then no iterate of $f$ preserves a singular holomorphic foliation.
\end{theorem}

\noindent
We refer the reader to \cite{FP,KPR}, and the references therein for more details
about the transformation of singular holomorphic foliations by rational maps.

The following lemma shows that the condition (D) can be readily satisfied:

\begin{lemma}
Suppose that $f$ has equal algebraic and topological degrees $d=\lambda_2\geq 2$.  Let $p\in\P^2$ be a point whose image $f(p)$ is a regular value of $f$, with $\lambda_2$ distinct $f$-preimages, all outside $\I$.  Let $L,L'$ be lines through $p$ that do not meet $\I$ and have distinct images $f(L)\neq f(L')$.  Then one of the two irreducible curves $f(L)$ or $f(L')$ is not a line.  In particular, for almost all lines $L\subset\P^2$, we have that $\deg f(L)~\geq~2$.
\end{lemma}

\begin{proof}
Suppose on the contrary that both images are lines.  Then the facts that $L\cap \I = \emptyset$ and that $f_*L = dL$ as divisors imply that $f:L\to f(L)$ is $d$-to-1.  Since $\lambda_2 = d$, it follows that $f^{-1}(f(L)) = L$.  Likewise $f^{-1}(f(L')) = L'$.  But then $f^{-1}(f(p))  = f^{-1}(f(L)\cap f(L')) \subset  L\cap L' = \{p\}$ which contradicts the hypothesis that $f(p)$ has $\lambda_2\geq 2$
distinct preimages.
\end{proof}

\begin{lemma}
\label{lem:curvegrows}
Suppose that conditions (A), (B) and (D) hold.  For any sufficiently small $\epsilon>0$ there exists $\delta>0$ such that for any $T\in \T(T_0,\delta)$, any irreducible algebraic curve $C\subset B_\epsilon(L_0)$, and any $n\geq 0$, the image $f_T^n(C)$ is irreducible, lies in $B_\epsilon(L_0)$ and has degree at least $2^n\deg C$.
\end{lemma}

\begin{proof}
Recall that $\epsilon$ is sufficiently small that $B_\epsilon(L_0)
\cap \I = \emptyset$.
Choose $\delta>0$ small enough that $T\in\T(T_0,\delta)$
guarantees that $f_T(\overline{B_\epsilon(L_0)}) \subset B_\epsilon(L_0)$.
This assures that $f_T^n|_{B_\epsilon(L_0)}: B_\epsilon(L_0) \rightarrow
B_\epsilon(L_0)$ is holomorphic for every $n \geq 1$.  Therefore, for any
irreducible curve $C \subset B_\epsilon(L_0)$ the irreducible curve $f_T^n(C)$ is just the usual set-theoretic image
of $C$ under $f_T^n$.
Since there are no exceptional curves for $f_T$ contained in $B_\epsilon(L_0)$, the image $f_T^n(C)$ is non-trivial.

It therefore suffices to show that $\deg f_T(C) \geq 2 \deg C$.  Since $\deg f(L_0)\geq 2$, we can find a line $L$  through $p_0$  such that $L\cap f(L_0)$ contains at least two points, and all of them are transverse intersections.  Shrinking $\epsilon$ (and therefore $\delta$) if necessary, we may assume that $f_T^{-1}(L)\cap B_\epsilon(L_0)$ is a union of $k\geq 2$ disks $\mathcal{D}_1,\dots,\mathcal{D}_k$ such that each is properly embedded in $B_\epsilon(L_0)$ and the images $f_T(\mathcal{D}_j)\subset L$ are mutually disjoint.  Since $C\subset B_\epsilon(L_0)$, we have on topological grounds that $C$ meets each disk $\mathcal{D}_j$ in at least $\deg C$ points (counted with multiplicity).  Hence $f(C)$ meets $L$ in at least $k\deg C$ points counted with multiplicity.  Thus $\deg f(C) \geq k\deg C \geq 2\deg C$.
\end{proof}

\begin{theorem} \label{THM_NO_FOLIATION}
Suppose that $\I\neq \emptyset$ and that conditions (A)-(D) hold.  Then for any $\delta>0$ small enough and any $T\in\T(T_0,\delta)$, there is no $f_T$-invariant singular holomorphic foliation of $\P^2$.
\end{theorem}

Note that $\I\neq \emptyset$ as soon as $\lambda_2 < d^2$, e.g. in the case of interest here $\lambda_2 = d\geq 2$.

\begin{proof}
Choose $\epsilon>0$ and then $\delta>0$ small enough so that all conclusions of
Proposition \ref{prop:twotraps} and Lemma \ref{lem:curvegrows} apply.  Fix any
$q\in \I$, any $T\in \mathcal{T}(T_0,\delta)$, and any non-trivial irreducible algebraic curve $C \subset f_T(q) \subset B_\epsilon(L_0)$.  
Lemma
\ref{lem:curvegrows} tells us that the forward images $f_T^{n-1}(C) \subset f_T^n(q)$ 
have degree tending to infinity with $n$.  Meanwhile, Proposition \ref{prop:twotraps}
implies that $f_T^{-n}(q)$ consists of exactly $\lambda_2^n$ regular points for~$f_T$.  Theorem \ref{THM_CONDITIONS_FOR_NO_FOLIATION} therefore guarantees that
there is no $f_T$-invariant foliation.
\end{proof}

\begin{corollary} 
\label{COR_GENERIC_NIF}
Let $f:\P^2\tto\P^2$ be a rational map with equal algebraic and topological
degrees $\lambda:= d(f) = \lambda_2(f)\geq 2$.  Then, the map $f_T:= T \circ f:
\P^2~\tto~\P^2$ admits no invariant foliation for any $T\in\aut(\P^2)$ outside
a countable union of proper algebraic subsets.
\end{corollary}

\begin{proof}
Let $q \in \I$ be any indeterminate point and let $C \subset f(q)$ be some chosen irreducible curve.
This induces a choice $C_{T} := T(C)$ of an irreducible component of $f_T(q)$ for each $T\in\aut(\P^2)$.
  In the proof of Theorem \ref{THM_NO_FOLIATION} we saw that there exists $S \in \aut(\P^2)$
so that 
\begin{enumerate}
\item 
for each $n \geq 1$, $f_S^{-n}(q)$ consists of $\lambda^n$ preimages each of which is a regular point for $f_S^n$, and 
\item for each $n \geq 1$, $f_S^{n-1}(C_{S}) \subset f_S^n(q)$ is an irreducible algebraic curve with $\deg(f_S^{n-1}(C_{S})) \rightarrow~\infty$, as $n \rightarrow \infty$.
\end{enumerate}
These two conditions imply the (weaker) criteria laid out in Theorem \ref{THM_CONDITIONS_FOR_NO_FOLIATION}.  To establish the present corollary, it will suffice to fix both $q\in\I$ and $n\in\N$ and show that the two conditions continue to hold when $S$ is replaced by any $T\in\aut(\P^2)$ outside of some algebraic subset.  Existence of $S$ implies we may take the subset to be proper.

Note that for any subvariety $V\subset\P^2$, the condition that $q\in f_T^n(V)$ amounts to an algebraic constraint on $T$.  Taking $V = \cdiv(f) \cup f_T(\I)$, we obtain that
$f_T^{-n}(q)$ consists of $\lambda_2^n$ distinct preimages, each of which is
a regular point for $f_T^n$, for all $T\in\aut(\P^2)$ outside some algebraic
subset.

For fixed $n$, the set of $T$ for which $f_T^{n-1}(C_T)$ is reducible is again
an algebraic subset of $\aut(\P^2)$.  The degree of $f_S^{n-1}(C_{S})$ is equal
to the number of set theoretic intersections $\#(f_S^{n-1}(C_{S})\cap L)$ of
$f_S^{n-1}(C_{S})$ with some line $L\subset\P^2$.   Therefore, for $T$ outside
of some proper algebraic set $\#(f_T^{n-1}(C_{T})\cap L)$ will be finite.
Moreover, it follows from the Weierstrass Preparation Theorem that as a
function of $T$, the quantity $\#(f_T^{n-1}(C_{T}) \cap L)$ is finite,
constant and maximal (in particular at least $\#(f_S^{n-1}(C_{S}) \cap L)$),
off another proper algebraic subset of $\aut(\P^2)$.  We conclude that $\deg
f_T^{n-1}(C_{T}) \to\infty$ for all $T$ outside a countable union of such
sets.  \end{proof}

\subsection{The saddle measure}

Recall that a rational map $f_0:\P^1\to\P^1$ is \emph{hyperbolic}, i.e. uniformly expanding on its Julia set, if all critical points of $f_0$ lie in basins of attracting periodic points.

\begin{proposition}
\label{prop:hypmap}
Let $L_0\subset \P^2$ be any non-exceptional line disjoint from $\I$ and $\p\in\P^2$ be any point outside $L_0$ and $f(L_0)$.  Then one can choose a linear map $T_0:\P^2\setminus\{\p\}\to L_0$ so that $T_0\circ f|_{L_0}$ is a degree $d$ hyperbolic rational self-map of $\P^1$ (i.e. of $L_0$).
\end{proposition}

\begin{proof}

Since we assume $L_0$ is disjoint from $\I$, the total
transform of $L_0$ under $f$ coincides with the proper transform $f(L_0)$.
Hence, the push forward divisor $f_*L_0$ has irreducible support equal to $f(L_0)$.
Because we have equivalence of divisors $f_*L_0 \sim dL_0$, it follows that
$f:L_0\to f(L_0)$ is $m$-to-$1$ where $m\deg f(L_0) = d$.  Now choose $\p$
outside $L_0\cup f(L_0)$.  Let $\Pi:\P^2\setminus\{\p\}\to~L_0$ denote the
central projection sending each line through $\p$ to its intersection with $L$.
The restriction $\Pi:f(L_0)\to L_0$ is $\deg f(L_0)$-to-1.  Hence the
restriction $\tilde f_0 := (\Pi\circ f)|_{L_0}$ is a degree $d$ rational map of
$\P^1$.  We choose the identification $L_0\cong \P^1$ so that $\infty$ is not
equal to $\tilde f_0(0)$ or to any critical value of $\tilde f_0$.  Let $f_0(z)
= \alpha \tilde f_0(z)$ for some $\alpha\in \C^*$.  Taking $|\alpha|$ small
enough, we obtain that all critical values of $f_0$ lie in a small disk
$\Delta\subset\C$ about $z=0$ and that $f_0(\overline{\Delta}) \subset \Delta$.
It follows that $\Delta$ is in the Fatou set of $f_0$ and therefore $f_0$ is
hyperbolic.  Let $A:L_0\to L_0$ denote the contraction $z\mapsto \alpha z$.
Then the (linear) map $T_0 := A\circ \Pi:\P^2\setminus\{\p\}\to L_0$ satisfies
the conclusion of the proposition.
\end{proof}

With $T_0$ as in the Proposition \ref{prop:hypmap}, we let $f_0$ denote the restriction $(T_0\circ f)|_{L_0}$
and $\J$ denote the Julia set of $f_0$.  Recall (see e.g. \cite{Jon2}) that a
\emph{history} of a point $z\in\J$ is a sequence $(z_{-n})_{n\geq 0}$
terminating with $z_0 = z$ and satisfying $f(z_{-(n+1)})~=~z_{-n}$ for all
$n\in\N$.  We let $\hat \J$ denote the set of all histories of points in $\J$
and endow $\hat \J$ with the product topology.  The \emph{natural extension} of
$f_0:\J\to\J$ is the homeomorphism $\hat f:\hat\J~\to~\hat\J$ given by $\hat
f(\hat z) = (f_0(z_{-n}))_{n\geq 0}$.  The canonical projection $\pi:~\hat \J~\to~\J$, given by $\pi(\hat z) = z_0$ is a semiconjugacy of $\hat f$ onto $f_0$
that induces a one-to-one entropy-preserving correspondence $\pi_*$ between
$\hat f$ invariant probability measures on $\hat \J$ and $f_0$ invariant
probability measures on $\J$.  The natural extension has the universal property
that any other semiconjugacy of an invertible system onto $f_0:\J\to\J$ factors
through $\pi$.  (See \cite[Ch. 1]{PU} for more details on the natural extension.)

\begin{theorem}
\label{thm:nuexists}
Suppose that conditions (A),(B), and (D)  hold and that \\ $T_0:\P^2~\setminus~\{\p\}~\to~L_0$ is chosen so that the induced one dimensional map $f_0 = T_0\circ f:\P^1\to \P^1$ is hyperbolic with Julia set $\J$.  Then for any $\epsilon>0$ small enough, there exists $\delta>0$ such that $T\in\T(T_0,\delta)$ implies that there is a unique $f_T$-invariant measure $\nu$ on $B_\epsilon(L_0)$ with entropy $\log d$. Moreover
\begin{itemize}
 \item[(a)] $f_T$ is uniformly hyperbolic of saddle type on $\supp\nu$.
 \item[(b)] The canonical projection $\pi:\hat \J\to \J$ factors $\pi = \pi_T\circ \hat\pi$ into semiconjugacies $\hat\pi_T:\hat\J \to \supp\nu$ of $\hat f$ onto $f_T$, and $\pi_T:\supp\nu\to \J$ of $f_T$ onto~$f_0$.
 \item[(c)] $\supp\nu$ is not contained in an algebraic curve.
 \item[(d)] Any point $p\in B_\epsilon(L_0)$ has forward orbit $(f_T^n(p))_{n\geq 0}$ asymptotic to either $\supp\nu$ or to one of finitely many attracting periodic cycles.
\end{itemize}
\end{theorem}

This theorem is closely related to papers on attractors  and attracting sets for
endomorphisms of $\P^2$ by Forn{\ae}ss-Weickert \cite{FW_ATTRACTORS},
Jonsson-Weickert \cite{JW_ATTRACTORS}, Forn{\ae}ss-Sibony \cite{FS_EXAMPLES},
Dinh \cite{DINH_ATTRACTORS} and Daurat \cite{DAURAT}.  It is also closely related to 
work on H\'enon and H\'enon-Like maps by Hubbard and Oberste-Vorth \cite{HOV2} and Dujardin \cite{DUJARDIN_HENON_LIKE}.  Regardless, we present a more or less self-contained proof here.

Recall that a sequence $(z_n)_{n\geq 0}\subset L_0$ is a $\delta$ pseudo-orbit for $f_0$ if $\dist(f(z_n),z_{n+1})~<~\delta$ for all $n\geq 0$.  

\begin{proof} 
Let $\epsilon>0$ be small enough that Proposition \ref{prop:twotraps} applies.  Because $f_0$ is hyperbolic, we may shrink $\epsilon$ still further if necessary and choose an open cover $\{\U^-,\U^+\}$ of $L_0$ and $\delta>0$ with the following properties: 
\begin{itemize}
\item all conclusions of Proposition \ref{prop:twotraps} apply;
\item $\J \subset \U^- \subset \overline{\U^-}\subset f_0(\U^-)$; 
\item $f_0(\overline{\U^+}) \subset \U^+\subset L_0\setminus \J$;
\item $f_0$ is uniformly expanding on $\U^-$ and uniformly contracting on $\U^+$;
\item for any $z\in L_0$, we have $f_0^n(z)\in \U^-$ for all $n\in \N$ if and only if $z\in\J$;
\item any $\delta$ pseudo-orbit in $\U^-$ is $\epsilon$-shadowed by a unique orbit in $\J$; 
\item any $\delta$ pseudo-orbit $(z_n)_{n\geq 0}$ with $z_0\in\U^+$ is completely contained in $\U^+$ and eventually contained in an $\epsilon$ neighborhood of some attracting cycle.
\end{itemize}
Strictly speaking, one must replace the Fubini-Study metric on $\U^\pm$ with some other metric to get the uniform expansion/contraction in the fourth item.  Alternatively, one can replace $f_0$ with a high enough iterate.  These things do not affect our arguments. 

Let $\Pi~:~\P^2\setminus\{p_0\}~\to~L_0$ be the central projection.  We extend $\U^-,\U^+$ to open subsets of $\P^2$, setting $\U^\pm_\epsilon := \Pi^{-1}(\U^\pm)\cap B_\epsilon(L_0)$. Fixing $T\in \T(T_0,\delta)$, we let
$$
\L_T := \{p\in B_\epsilon(L_0):f_T^n(p) \in \U^-_\epsilon \text{ for all } n\geq 0\}. 
$$
For any point $p\in \L_T$, the sequence  $(\Pi\circ f_T^n(p))_{n\geq
0}$ is a $\delta$ pseudo-orbit for $f_0$ that is completely contained in
$\U^-$, so there is a unique point $\pi_T(p) \in \J$ whose $f_0$-orbit $\delta$
shadows that of $p$.  Clearly $\pi_T\circ f_T = f_0\circ \pi_T$.  The fact that
$f_0$ is uniformly expanding on $\U^-$ guarantees that the semiconjugacy
$\pi_T:\L_T\to\J$ is continuous.

Similarly, for any point $p\notin\L_T$, the sequence  $(\Pi\circ f_T^n(p))_{n\geq 0}$ is a $\delta$ pseudo-orbit for $f_0$ eventually contained $\U^+$ and therefore eventually contained in an $\epsilon$ neighborhood of some attracting periodic cycle $(z_n)_{n=0}^{N-1}$.  Since $f_0$ is contracting on $\U^+$, $f_T$ is contracting on $\U^+_\epsilon$.  It follows that there is a unique attracting periodic cycle $(p_n)_{n=0}^{N-1}$ for $f_T$ with $d(p_n,z_n) < \epsilon$ for all $0\leq n\leq N-1$ and that the orbit of $p$ is asymptotic to this cycle.  That is, every point $p\in B_\epsilon(L_0)\setminus \L_T$ has a forward orbit asymptotic to one of finitely many attracting cycles.

\begin{lemma}
\label{lem:lam}  For small enough $\delta>0$ and $T\in\T(\pi,\delta)$, the map $\pi_T:\L_T\to \J$ is a fibration of $\L_T$ by complex disks that are properly embedded in $B_\epsilon(L_0)$, and the maps $f_T:\pi_T^{-1}(z) \to \pi_T^{-1}(f_0(z))$ are uniformly contracting.  In particular $f_T$ is uniformly hyperbolic of saddle type on $\L_T$.
\end{lemma}

The proof is a small variation on that of the Hadamard-Perron Theorem (see e.g. \cite[Thm. 6.2.8]{KH}), helped along a bit by the holomorphic context.  We sketch it for the sake of completeness.

\begin{proof}
Replacing $f_T$ and $f_0$ with iterates for the moment, we may assume that
$f_0$ expands the Fubini-Study metric on $\U^-$ by a uniform factor of $4$.
Let $\Pi:\P^2\setminus\{\p\}\to L_0$ be the central projection.  Given $z\in \J$,
let $D_n, D_n' \subset L$ be the disks of radius $\delta$, $\delta/2$ about
$f_0^n(z)$.  Thus $f_0$ maps $D_{n-1}'$ biholomorphically onto a
neighborhood of $\overline{D_n}$.

Setting $P_n = \Pi^{-1}(D_n)\cap B_\epsilon(L_0)$, we call a disk $\Delta$ in $P_n$ \emph{vertical} if $\Delta$ meets $L_0$ transversely in a single point and $\partial\Delta\subset \partial B_\epsilon(L_0)$.  Provided $\delta > 0$ is small enough, $f_T^{-1}(\Delta)$ meets $P_{n-1}$ in a unique vertical disk $f_T^\sharp(\Delta)$.  If, e.g. by linear projection from a point in $L_0$ well outside $\J$, we regard all vertical disks in $P_n$ as graphs over the `central disk' $\Pi^{-1}(f_0^n(z))$, then $f_T^\sharp$ uniformly contracts distances between graphs.  It follows that if we choose vertical disks $\Delta_n\subset P_n$ for each $n\in\N$, then $(f_T^n)^\sharp(\Delta_n)$ converges uniformly to a vertical disk $\Delta_z \subset P_0$.  Necessarily $f^n(\Delta_z)\subset P_n$ for all $n\in\N$ so that $\pi_T(\Delta_z) = z$.  Since the $P_n$ are foliated by vertical disks, it follows that in fact $\Delta_z = \pi_T^{-1}(z)$, i.e. any point $p\in P_0\setminus \Delta_z$ satisfies $f_T^n(p)\notin P_n$ for some $n$ large enough.  

Since $f_{T_0}$ contracts all vertical disks to $L_0$, it follows (for small enough $\delta$) that $f_T(\Delta_z) \subset B_{\epsilon/2}(L_0)\cap \Delta_{f(z)}$.  So by the Schwarz Lemma $f_T:\Delta_z\to \Delta_{f(z)}$ is uniformly contracting.  Since $f_0$ is expanding on $\J$, $f_T$ must uniformly expand the distance between distinct `stable' disks $\Delta_z,\Delta_{z'}$.  Hence $f_T$ is hyperbolic of saddle type on $\L_T$.
\end{proof}

Let $\Omega_T := \bigcap f_T^n(\L_T)$.  Since $\L_T$ is closed in $B_\epsilon(L_0)$ and $f(\overline{B_\epsilon(L_0)}) \subset B_\epsilon(L_0)$, it follows that $\Omega_T$ is compact and non-empty and satisfies $f_T(\Omega_T) = \Omega_T$.  The previous lemma and $f_0^{-1}(\J) = \J$ guarantee that $\pi_T:\Omega_T \to \J$ is surjective.  Since $f_T$ is uniformly contracting along fibers of $\pi_T$, we have for each history $\hat z\in\hat\J$ that the intersection $\bigcap_{n\geq 0} \overline{f_T^n(\pi_T^{-1}(z_{-n}))}$ is a single point $\hat\pi_T(\hat z)$.  More or less by construction, the resulting map $\hat\pi_T:\hat\J \to \Omega_T$ semiconjugates $\hat f$ to $f_T$ and satisfies $\pi = \pi_T\circ\hat\pi_T(\hat z)$.  It is continuous because $f_T$ is continuous and surjective because $f_T(\Omega_T) = \Omega_T$.    

Thus, by the universal property of the natural extension we have that $\hat f:\hat J\to\hat J$ is also the natural extension of $f_T:\Omega_T\to\Omega_T$.  The unique invariant probability measure $\check \nu$ of maximal entropy for $f_0:\J \to \J$ lifts to a unique measure of maximal entropy $\hat \nu$ for $\hat f:\hat J\to\hat J$.  Pushing forward to $\nu := \hat\pi_{T*}\hat \nu$ gives the unique measure of maximal entropy $\log d$ for $f_T:\Omega_T\to\Omega_T$.

To complete the proof of Theorem \ref{thm:nuexists} it remains to show that $\supp\nu$ is not contained in an algebraic curve.  Suppose, in order to reach a contradiction, that $\supp\nu$ is contained in an algebraic curve $C$.  Since $\nu$ is ergodic with infinite support, we can suppose each irreducible component of $C$ intersects $\supp \nu$ in an infinite set with positive measure, and that these components are permuted cyclically by $f_T$. 

It follows from Lemma \ref{lem:curvegrows} that no irreducible component of $C$ is contained in $B_\epsilon(L_0)$.
Slightly shrinking $\epsilon$ if necessary, we may assume that $B_\epsilon(L_0)$ omits at least three points of each such component.  So $C\cap B_\epsilon(L_0)$ is forward invariant by $f_T$, and its normalization $S$ is a Riemann surface whose connected components are all hyperbolic.  Hence the restriction of $f_T$ to $C\cap B_\epsilon(L_0)$ lifts to a
map on $S$ which is distance non-increasing in the hyperbolic metric.  In particular the restriction of
$f_T$ to $C\cap B_\epsilon(L_0)$ has topological entropy zero.  This
contradicts the fact that $\nu$ is an $f_T$-invariant measure of positive
entropy supported on  $C\cap B_\epsilon(L_0)$.  
\end{proof}

\subsection{The repelling measure}

Recall \cite{KITCHENS} that for any positive integer $k$, the \emph{one-sided $k$-shift} is the set $\Sigma_k$ of all sequences $(i_n)_{n\geq 0}$ that take values $i_n\in\{0,\dots,k~-~1\}$, together with the map $\sigma:\Sigma_k\to\Sigma_k$, $\sigma:(i_n) \mapsto (i_{n+1})$.  The one-sided $k$-shift has topological entropy $\log k$ and admits a unique invariant measure that achieves this entropy.

\begin{theorem}
\label{thm:muexists}
Suppose that conditions (A),(B) and (C) hold.  Then for $\epsilon>0$ small enough there exists $\delta>0$ such that for any $T\in \T(T_0,\delta)$
\begin{itemize}
 \item $\A := \bigcap_{n\geq 0} f_T^{-n}(B_\epsilon(f^{-1}(p_0)))$ is a Cantor set totally invariant by $f_T$;
 \item $f_T$ is uniformly expanding on $\A$;
 \item $f_T|_\A$ is topologically conjugate to the one-sided $\lambda_2$-shift;
 \item in particular, there is a unique $f_T$-invariant measure $\mu$ with $h_\mu(f_T) =
\log\lambda_2$ having 
$\supp\mu \subset B_\epsilon(f^{-1}(p_0))$.
\end{itemize}
\end{theorem}

\begin{proof}
Choose $\epsilon>0$ and then $\delta>0$ small enough that all conclusions of Proposition~\ref{prop:twotraps} apply.  Since $p_0$ is a regular value of $f$, we may further assume that $B_\epsilon(f^{-1}(p_0))$ is a disjoint union of $\lambda_2$ open balls, each centered at a preimage $p_j$ of $p_0$ and mapped biholomorphically by $f$ onto a neighborhood of $p_0$.  In particular, $B_\epsilon(f^{-1}(p_0))$ is Kobayashi complete hyperbolic.  (See \cite{KRANTZ} for a gentle introduction to the Kobayashi metric and \cite{LANG} for more details.)

Shrinking $\delta$ if necessary, we may assume that $T$ maps the complement of 
$f(B_\epsilon(p_j))$ into $B_\epsilon(L_0)$, so that $f_T(B_\epsilon(p_j))$
contains $\overline{B_\epsilon(f^{-1}(p_0))}$. It follows that $f$ uniformly
expands the Kobayashi distance on $B_\epsilon(f^{-1}(p_0))$.

Standard arguments now tell us that if we assign to each point $p\in \A$ its  `itinerary' $\iota(p) = (i_n)$, where $f_T^n(p) \in B_\epsilon(p_{i_n})$, then the resulting map $\iota:\A \to \Sigma_{\lambda_2}$ is a homeomorphism that conjugates $f_T|\A$ to the shift map $\sigma$.  Pulling back the unique measure of maximal entropy for $\sigma$ gives us the measure $\mu$ in the final conclusion of the theorem.
\end{proof}

In order to show that the repelling measure $\mu$ is not contained in an algebraic curve, we first indicate some useful refinements of Proposition \ref{prop:hypmap}.

\begin{proposition}
\label{prop:hypmapplus}
Suppose $p_0$ and $L_0$ satisfy $p_0 \not \in f(L_0)$ and  Properties (B-D) and let $L_1$ be some chosen line
through two distinct points of $f^{-1}(p_0)$.
One can choose a linear map
$T_0:\P^2\setminus\{\p\}\to L_0$ so that $T_0\circ f|_{L_0}$ is a degree $d$ hyperbolic rational self-map of $L_0$ and, additionally, there exists a disk
$\Delta\subset L_0$ satisfying
\begin{itemize}
\item $\Delta$ is disjoint from the critical set of $f$ and from $L_1 \cap L_0$,
\item $\Delta$ contains
$T_0\circ f(\overline{\Delta})$, $T_0\circ f(\E)$, and $T_0\circ f(L_1 \cap L_0)$, and
\item  each irreducible component of  $\E$ has infinite forward orbit under  $T_0 \circ f$.
\end{itemize}
\end{proposition}

\begin{proof}
In the proof of Proposition \ref{prop:hypmap} we can (additionally)
identify $0$ and $\infty$ with regular points of $f$ different from $L_1\cap L_0$. 
All conclusions of Proposition \ref{prop:hypmap}
and the first two conclusions above follow by taking both $|\alpha|$ and $\Delta$
small enough. 

Let $T_0$ be the map obtained in the previous paragraph.
Concerning the third conclusion, note that $T_0 \circ f$ is injective on $\Delta$ and has a unique fixed point $\eta$,
so it suffices to arrange that no irreducible component $E$ of $\E$ has image $T_0 \circ f(E)~=~\eta$.
This can be done without affecting any of the previous properties by an arbitrarily small perturbation of $T_0$.
\end{proof}

\begin{theorem}
\label{THM_MU_NOT_1D}
Suppose in Theorem \ref{thm:muexists} that the rational map $f$ has equal degrees $d = \lambda_2$.  If $\epsilon>0$ and $\delta>0$ are small enough and $T_0$ is chosen to satisfy the hypotheses of Propositions \ref{prop:hypmap} and \ref{prop:hypmapplus}, then there is no algebraic curve in $\P^2$ that contains $\supp\mu$.  
\end{theorem}

We begin proving this by drastically narrowing down the possibilities for an algebraic curve that contains $\supp\mu$.

\begin{lemma}
\label{LEM_SMOOTH_CURVE}
Suppose under the hypotheses of Theorem \ref{THM_MU_NOT_1D} that $T\in\T(T_0,\delta)$ and $C\subset\P^2$ is an algebraic curve containing $\supp\mu$, where $\mu$ is the repelling measure for $f_T$.  Then $\E(f) = \E(f_T) \neq \emptyset$ and we may assume that
\begin{itemize}
 \item $C$ is smooth and rational, in particular irreducible.
 \item $C = f_T^{-1}(C)$ is totally invariant.
 \item $\I(f_T^n) \cap C = \emptyset$ for all $n\geq 1$.
 \item The forward orbit $f_T^n(\E(f))$ is contained in $C$.
\end{itemize}
\end{lemma}

\begin{proof}
The measure of maximal entropy on the $\lambda_2$-shift is mixing, hence so is $\mu$.  We may therefore assume that $C$ is irreducible and invariant, i.e. $f_T(C) = C$.  Since $\mu$ has entropy $\log\lambda_2$ and totally disconnected support, it further follows that $C$ is rational and that $f_T|_C$ is (at least and therefore exactly) $\lambda_2$-to-1.

In particular, $C$ is \emph{totally} invariant, i.e. $f_{T}^{-1}(C) = C$.  Furthermore, 
\begin{align*}
\lambda_2\deg C\leq \deg f_{T*}C = d\deg C = \lambda_2\deg C
\end{align*}
 implies that $f_{T*} C = \lambda_2 C$ and therefore that $C$ contains no indeterminate points for $f_T$.  Invariance of $C$ implies that it contains no indeterminate points for $f_T^n$ for any $n\geq 1$.

Since $\supp\mu$ is disjoint from the critical set of $f_T$, we see that $f_T$ is not critical along $C$.  Hence $f_T^* C = C + E_C$ where $E_C$ is a non-trivial effective divisor supported on the exceptional set $\E$.  In particular $\E\neq \emptyset$.  Each irreducible component $E$ of $\E$ must meet $C$ at some non-indeterminate point $p$, so it follows that $f_T^n(E) = f_T^n(p) \subset C$ for all $n\geq 1$.  

It remains to prove that $C$ is smooth.  Suppose that $p$ is a singular point of $C$ and that $q\in\P^2$ is any $f_T$-preimage of $p$.  Any local defining function $\psi$ for $C$ vanishes to order at least $2$ at $p$.  Hence the local defining function $\psi\circ f_T$ for $f_T^*C$ vanishes to order at least $2$ at $q$, i.e $f_T^*C$ is singular at $q$.  Since $f_T^*C = C + E_C$, we infer that either $C$ is singular at $q$ or $q\in\E$.  In the first case, we repeat the argument with $q$ in place of $p$, etc.  Since $C$ has only finitely many singular points, we will eventually find ourselves in the second case.  In short, \emph{any} backward orbit of $p$ contains a point in $f_T(\E)\cap C$.

It follows from Proposition \ref{prop:hypmapplus} that for $T$ close enough to
$T_0$ any irreducible component of $E$ of $\E$ has $f_T(E)$
in a forward invariant neighborhood $W$ of the
disk $\Delta$.  Moreover, $f_T|_W$ is a biholomorphism onto its
image with a unique fixed point.  By the third part of Proposition \ref{prop:hypmapplus}, this fixed point is
different from $f_T(E)$.
Therefore, $p$ has an infinite forward orbit along which $C$ would be singular, giving a contradiction.
We conclude that $C$ is smooth.
\end{proof}

\begin{proof}[Proof of Theorem \ref{THM_MU_NOT_1D}]
If the theorem fails, then there is a sequence $T_i \in \aut(\P^2)$, with
$T_i \rightarrow T_0$, for which the repelling invariant measure $\mu_i$
associated to $f_{T_i}$  is contained in an algebraic curve $C_i$.  By
Lemma \ref{LEM_SMOOTH_CURVE}, each $C_i$ is smooth and rational.  In particular each has degree one or two.
Refining, we may assume that $\deg C_i$ is independent of $i$ and that $C_i \to C_\infty$, where $C_\infty$ is a (possibly reducible) divisor of the same degree, and the convergence can be understood to take place with respect to coefficients of the homogeneous defining polynomials for $C_i$.  

We claim that $L_0$ is an irreducible component of $C_\infty$.  Indeed, $C_i$ contains the forward orbit of the exceptional set $\E$ by $f_{T_i}$.  As $i\to\infty$ this forward orbit converges to that of $\E\cap L_0$ by $T_0\circ f$.  
By Proposition \ref{prop:hypmapplus}, the latter is an infinite subset of $L_0$, so that $L_0$ must be contained in $C_\infty$.  So the claim follows from the fact two algebraic curves with infinitely many points in common must share an irreducible component.

Since $\supp\mu\subset C_i$ converges to $f^{-1}(p_0)$ as $i\to\infty$, it further follows that $f^{-1}(p_0)\subset C_\infty$.  By hypothesis $f^{-1}(p_0)$ contains at least two points, all distinct from $L_0$.
Hence the only possibility here is that $C_\infty = L_0\cup L_1$, where $L_1\neq L_0$ contains the entire preimage $f^{-1}(p_0)$ of $p_0$.  Proposition \ref{prop:hypmapplus} implies that $L_1\cap L_0$ is a point in the Fatou set of the one dimensional map $f_0 := T_0\circ f:L_0 \to L_0$.  As a rational map of degree $d>1$, the map $f_0$ has at least one repelling fixed point $s\in L_0$.  Since it lies in the Julia set $\J$ of $f_0$, the point $s$ differs from $L_0\cap L_1$.

Let $\pi_{T_i}:\L_{T_i}\to \J$ be as in the proof of Theorem \ref{thm:nuexists}.  By Lemma \ref{lem:lam}, the sets $D_i = \pi_{T_i}^{-1}(s)$ are complex disks properly embedded in $B(L_0,\epsilon)$, forward invariant and contracted by $f_{T_i}$.  The point $s_i := \bigcap f_{T_i}^n(D_i)$ is fixed of saddle type for $f_{T_i}$.
As $i\to\infty$, we have $s_i\to s$ and $D_i \to B(L_0,\epsilon)\cap \Pi^{-1}(s)$ uniformly.  Since $C_i\to L_0\cup L_1$, we also have for large $i$ that the central projection $\Pi:\P^2\setminus\{p_0\}$ onto $L_0$ restricts to an injective map of $C_i\cap B(L_0,\epsilon)$ onto $L_0$ minus a small neighborhood of $L_0\cap L_1$.  In particular $C_i\cap D_i\neq \emptyset$ for $i$ large enough.  

Since $C_i\cap D_i$ is forward invariant and closed, it follows that $C_i$ contains the saddle point $s_i$.  Since $C_i$ also contains a repelling fixed point near each preimage of $p_0$ and the unique attracting point in $B(L_0,\epsilon)$, we see that $f|_{C_i}$ has at least $d + 2$ fixed points.  However, a degree $d$ map of $\P^1$ has only $d+1$ fixed points.  This contradiction concludes the proof.
\end{proof}

We have now established the main results stated at the beginning of this paper.
That is, if equality $\lambda_2 = d$ holds for topological and algebraic degrees of
$f$, then Proposition \ref{PROP_AS} and Theorems \ref{THM_NO_FOLIATION},
\ref{thm:nuexists}, \ref{thm:muexists} and \ref{THM_MU_NOT_1D} together give Theorem \mainthma; and Theorem {\mainthmb } follows from Corollaries \ref{COR_AS} and \ref{COR_GENERIC_NIF}.

\section{A specific example}\label{SEC_PERTURBED_CHEBYSHEV}
\label{section:unperturbed}

In the remainder of this article, we consider some rather restricted linear perturbations of a particular rational map $f: \P^2 \dashrightarrow \P^2$ that is closely related to the one variable quadratic Chebyshev map and was studied in detail by Han Liu in his PhD thesis \cite{han}.  Specifically $f = g \circ h$ where
\begin{equation}
\begin{array}{rcl}
g[x_1:x_2:x_3] & = & [x_1^2:x_2^2:x_3^2], \quad \mbox{and} \\
h[x_1:x_2:x_3] & = & [x_1(-x_1+x_2+x_3):x_2(x_1-x_2+x_3):x_3(x_1+x_2-x_3)].
\end{array}
\end{equation}
Note that $h$ is a birational involution, linearly conjugate to the standard Cremona involution $[x_1,x_2,x_3]\mapsto [x_2x_3,x_3x_1,x_1x_2]$ via the (unique) automorphism of $\P^2$ that sends the points $[1,0,0],[0,1,0],[0,0,1],[1,1,1]$ to $[0,1,1], [1,0,1], [1,1,0],$\\ $[1,1,1]$, respectively.  Hence the indeterminacy set of $h$ is
$$
\I(h) = \{a_1,a_2,a_3\} := \{[0,1,1],[1,0,1],[0,1,1]\},
$$ 
and the exceptional set $\E(h)$ consists of the three lines $A_1,A_2,A_3$ joining these points.  Specifically $h(A_i) = a_i$ and vice versa, where $A_i$ is the line joining the pair $\I(h)\setminus\{a_i\}$.

Both maps $g$ and $h$ preserve the rational two form $\eta$ given in the affine
coordinates $(x,y)\mapsto [x,y,1]$ by $\eta = \frac{dx\wedge dy}{xy}$.  That
is, $g^*\eta = 4\eta$ and $h^*\eta = \eta$.  Hence $f^*\eta = (4\cdot 1)\eta$.
Plane rational maps that preserve meromorphic two forms (up to a multiplicative factor) are considered at length in
\cite{DL}. 
The map $f$ considered here is one of the simplest
instances we know of a non-invertible rational map that preserves a two form
and has non-obvious dynamics.  These dynamics are completely described in
\cite{han}, but here we pay attention only to aspects relevant to the theme of
this article.

Several elementary properties of $f$ are readily inferred from those of $g$ and $h$.

\begin{proposition}\label{PROP:BASIC_CHEB}
$f$ is a dominant rational map with topological and algebraic degrees $\lambda_2(f) = d(f) = 4$.  Moreover,
\begin{enumerate}
\item $f$ is symmetric in the homogeneous coordinates $x_1,x_2,x_3$ and preserves (modulo indeterminate points) the real slice $\RP^2 := \{[x_1,x_2,x_3]\in\P^2: x_j\in\R\}$.  
\item The poles $X_j := \{x_j=0\}$ of $\eta$ are each totally invariant by $f$.  These lines are also the ramification locus of $f$, and the restriction $f:X_j\to X_j$ to any one of them, when expressed in the two non-vanishing homogeneous coordinates, is the one variable map $z\mapsto z^2$.
\item The points $\{e_1,e_2,e_3\} := [1,0,0],[0,1,0],[0,0,1]$ are each fixed and superattracting.  %The basins of each of these points is connected.
\item The only other (non-indeterminate) fixed point of $f$ is $[1,1,1]$, which is repelling.
\item The exceptional and indeterminacy loci of $f$ coincide with those of $h$, and $f(A_j) = a_j$.  In particular, $f$ is not algebraically stable on $\P^2$.
\item The critical divisor of $f$ is reduced with degree six.  Specifically $\cdiv(f) = \sum_i A_i + \sum_j X_j$.  
\end{enumerate}
\end{proposition}

In order to understand the real dynamics of $f$, it is convenient to employ affine coordinates that are adapted to emphasize the symmetry of $f$ with respect to the homogeneous variables.  Specifically, in what follows \emph{adapted affine coordinates} will mean affine coordinates $(x',y')$ on $\P^2\setminus L_0$, where $L_0=\{x_1+x_2+x_3=0\}$, that identify the repelling fixed point $[1,1,1]$ with $(0,0)$ and the superattracting fixed points $[1:0:0],[0:1:0],[0:0:1]$ with vertices of an equilateral triangle centered at $(0,0)$.  In these coordinates, Figure \ref{fig:basins} shows the real points in the basins associated to each superattracting point $e_1,e_2,e_3$.  This picture reveals some interesting aspects of the dynamics of $f$.  The complement of the (closures of the) basins is (apparently) the open disk $\U$ inscribed in the triangle $\{x_1x_2x_3=0\}$.  One can check that $\partial\U = \R\P^2\cap Q$, where $Q$ is the algebraic curve defined by $\rho(x_1,x_2,x_3) = x_1^2 + x_2^2+x_3^2 - 2(x_1x_2 + x_2x_3 + x_3x_1) = 0$) and that $Q\cap\{x_1x_2x_3=0\} = \I$, each intersection being a tangency.  

\begin{figure}
\centering
\includegraphics[width=90mm]{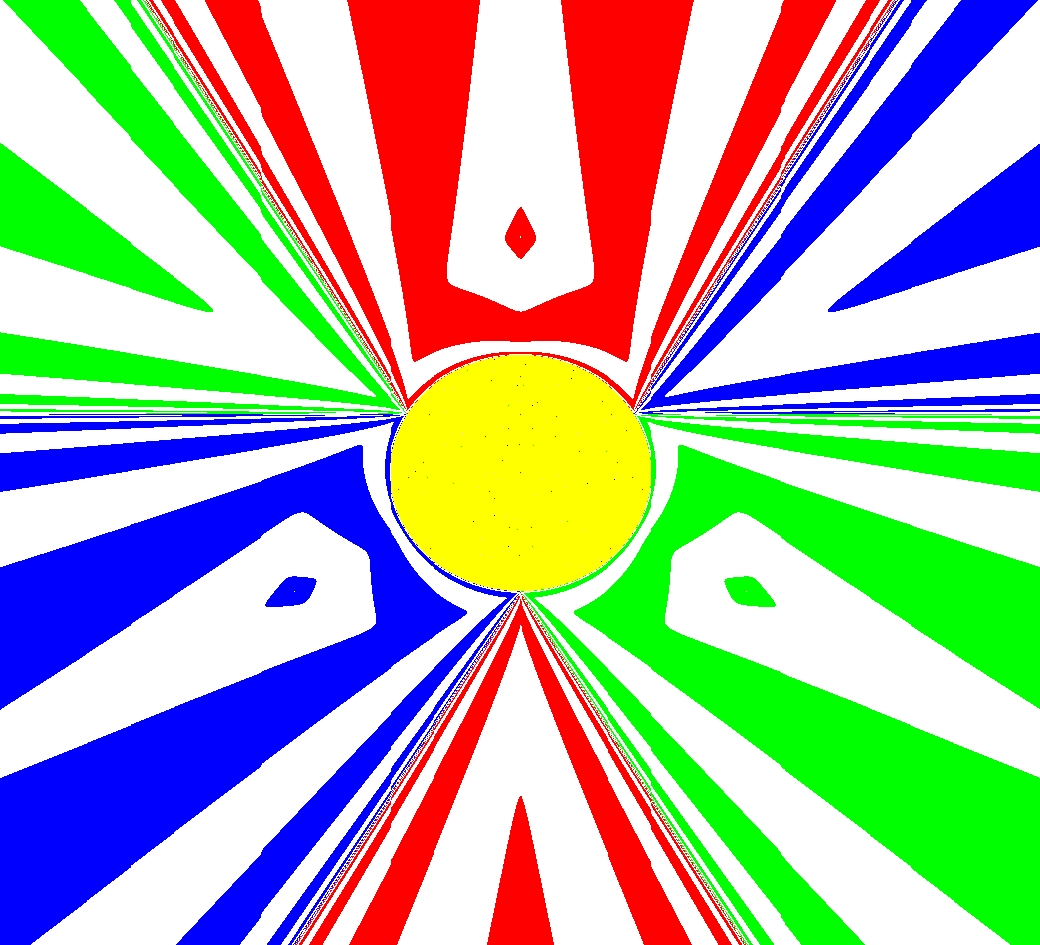}
\caption{Real parts of the three superattracting basins shown in adapted affine coordinates.  Alternating white/colored bands indicate the amount of time it takes for a point to get close to a superattracting point.  The disk $\U$ in the center is the complement of the basin closures.  Note also that the common boundary of each pair of basins is contained in an exceptional line $A_j$, and that $\I\subset\partial\U$ is the set where all three basins meet.  \label{fig:basins}}
\end{figure}

\begin{proposition} 
\label{prop:Q}
The conic curve $Q$ and the region $\U$ have the following properties.
\begin{enumerate}
\item $f(a_j) = Q$ for each $a_j\in \mathcal{I}(f)$. 
\item $f(p) = p$ for every $p\in Q\setminus \mathcal{I}(f)$.
\item $\U$ is totally invariant modulo the lines $A_j$, i.e. for every $p\notin A_1\cup A_2\cup A_3$, we have $p\in \U$ if and only if $f(p) \in \U$. 
\end{enumerate}
\end{proposition}

\begin{proof}
We have that $f(a_j) = g(h(a_j)) = g(A_j)$ is a conic curve tangent to each coordinate axis $X_i$ at the point $g(A_j\cap X_i) = \mathcal{I}(f)\cap X_i$, independent of $j$.  Counting the number of conditions that the tangencies impose on the defining polynomial, one finds there is only one such conic, so we must have $f(a_j) = Q$.  

To see that $Q$ is totally invariant by $f$, note that $g^*Q$ is a divisor of degree 4 with defining polynomial that is symmetric in the homogeneous coordinates $x_1,x_2,x_3$ and that the support of $g^*Q$ includes (by the previous paragraph) the lines $A_1,A_2,A_3$.   It follows that $g^*Q = A_1+A_2 + A_3 + L_0$, where $L_0 = \{x_0+x_1+x_1=0\}$ is the `line at infinity' in Figure \ref{fig:basins}.   Now $h^{-1} = h$ collapses each line $A_j$ to a point, and $h^*L_0$ is an effective divisor of degree two passing through each point in $\mathcal{I}(f)$ and (again) symmetric in $x_1,x_2,x_3$.  It follows that $f^{-1}(Q) = h^{-1}(L_0) = \mathop{\supp} h^*L_0 = Q.$

Now we argue that $f|_Q = \mathrm{id}$.  Since $h$ is an involution mapping $L_0$ to $Q$, we have that $h$ maps $Q$ bijectively onto $L_0$.   Thus, $g(L_0) = Q$, and $g_*L_0 = k Q$ where $k$ is the topological degree of the restriction $g|_{L_0}$.  Since $g(L_0)$ is a divisor of degree two, we have $k=1$, i.e. $g:L_0 \to Q$ is also injective.  Hence $f|_Q$ is an automorphism of a rational curve.  Symmetry with respect to homogeneous coordinates dictates that $f$ fixes each of the three points on $Q$ where two of the three homogeneous coordinates agree.  Hence the restricted map, a linear fractional transformation fixing three points, must be the identity.

Using the fact that $(1,1)\in \U$ one sees that 
\begin{align*}
\U = \{(x,y)\in \R^2:\rho_\mathrm{aff}(x,y) < 0\},
\end{align*}
 where $(x,y) = [x,y,1]$ are (non-adapted) affine coordinates and $\rho_\mathrm{aff}(x,y) = \rho(x,y,1)$.  Continuing the computation of the pullback of $Q$ begun above, we arrive at
$$
f^*Q = h^*(A_1+A_2+A_3+L_0) = 2(A_1+A_2+A_3) + Q.
$$
Hence $\rho_\mathrm{aff}\circ f = c \rho_\mathrm{aff} \rho_1^2\rho_2^2\rho_3^2$ where $\rho_j$ is an affine defining function for the line $A_j$ and $c$ is a constant.  Applying this formula to $(1,1) = f(1,1)$, we see that $c > 0$.  Hence for any point $p\notin A_1\cup A_2\cup A_3$, we see that $\rho_\mathrm{aff}(p)$ has the same sign as $\rho_\mathrm{aff}(f(p))$.  This proves the final assertion in the proposition.
\end{proof}

\begin{corollary}
\label{cor:eqmeasure} $f$ has topological entropy equal to $\log 4$, and there is a unique measure of maximal entropy for $f$ given by $\mu = \frac{1}{2\pi^2}  {\bm 1}_{\U}  \eta$.  In particular,  $h_{top}(f|_{\RP^2}) = h_{top}(f) = \log 4$ and repelling periodic points of $f$ are dense in $\U$.  
\end{corollary}

\begin{proof}
Note that the two form $\eta$ naturally defines a positive measure on $\U$.  A routine computation shows that this measure has finite mass equal to $2\pi^2$.  Since $\U$ is totally invariant by $f$ and $f^*\eta = 4\eta$, it follows that $f^*\mu = 4\mu = \lambda_2(f)\mu$.  Hence by \cite{PARRY} $h_\mu(f) = \log 4$.  But the variational principal and Gromov's bound \eqref{eqn:gromovbd} for the entropy of a rational map tell us that
$$
h_{\mu}(f) \leq h_{top}(f) \leq \log\max\{\lambda_1(f),\lambda_2(f)\} \leq \log\max\{d(f),\lambda_2(f)\} = \log 4
$$
so that $h_{top}(f) =\log 4$, too.  In fact, $\lambda_1(f) < d(f)$ by Part (5) of Proposition \ref{PROP:BASIC_CHEB}, so uniqueness of the measure of maximal entropy and density of repelling cycles follows from the main results of \cite{Guedj}.
\end{proof}

\begin{remark}  It is shown in \cite{han} that $\lambda_1(f) = 2$.
\end{remark}

The lines $A_j$ partition the disk $\U$ into four simply connected open sets $\U_i$, $i~=~0,1,2,3$.  We index these so that $\U_0$ denotes the center triangle, and $\U_i$, $i=1,2,3$, denotes the set bounded by $A_i$ and $Q$.  The partition $\left\{\overline{\U_i}\right\}_{i=0}^3$ maps forward well.

\begin{figure}
\centering
\includegraphics[width=90mm]{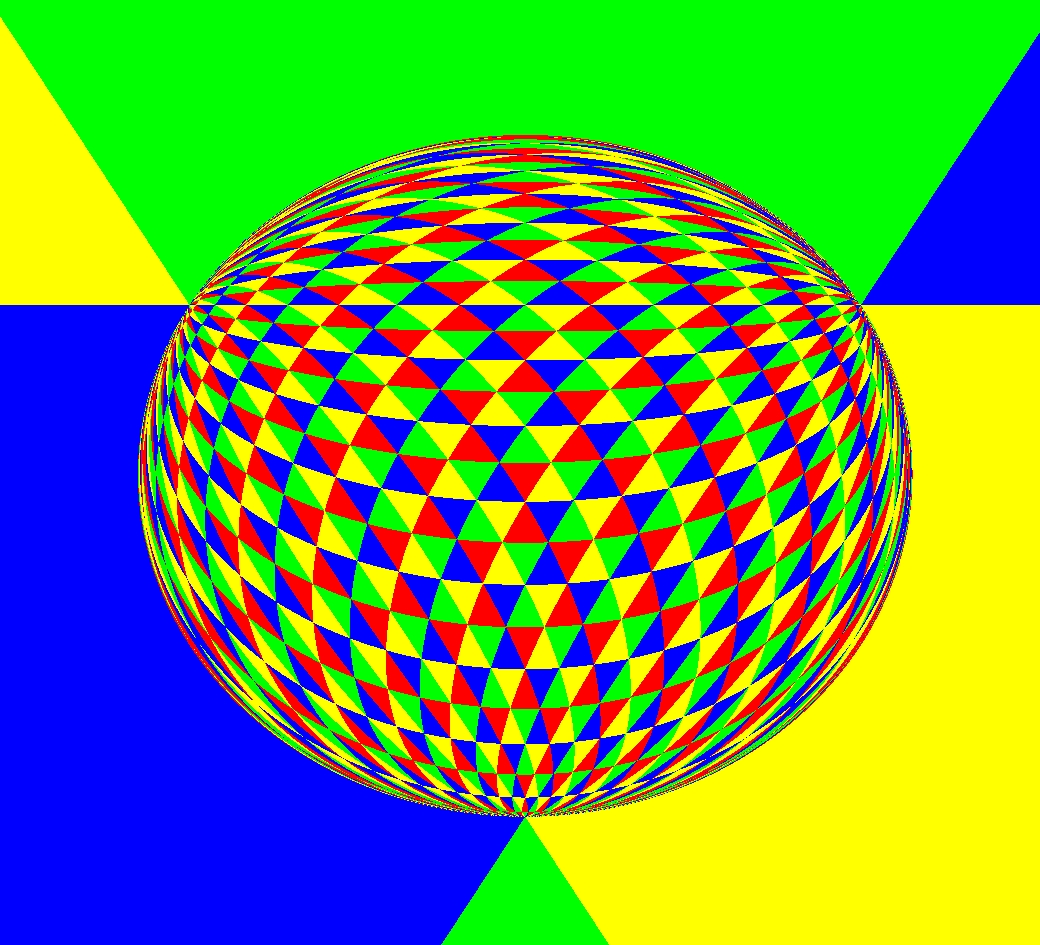}
\caption{Dynamics on $\U$. Each point $p$ is colored according to the connected component of $\U\setminus \{A_1,A_2,A_3\}$ that contains $f^4(p)$.  \label{fig:partition}}
\end{figure}

\begin{proposition}\label{prop:partition}
$f$ maps each region $\U_j$ homeomorphically onto $\U$.  
\end{proposition}

\begin{proof}
The last conclusion of Proposition \ref{prop:Q} tells us that $f(\U_j) \subset \U$.  Since $f$ has topological degree $4$, it will suffice to show that $\U\subset f(\U_j)$ for $j=0,1,2,3$.

Consider first the image of $\U_0$.  From the fact that $(x,y)\mapsto (1/x,1/y)$ preserves the first quadrant, it follows that $h(\U_0) = \U_0$.  One then sees that $g$ maps $\partial \U_0$ homeomorphically to $Q$.  Indeed $A_3 \cap \partial \U_0$ is the line segment in $\R^2$ joining $(1,0)$ to $(0,1)$), and one verifies easily that $s$ maps  this segment homeomorphically onto the portion of $Q$ joining $(1,0)$ to $(0,1)$.  Similar observations apply to the the other two sides of $\U_0$.   It follows that $f(\U_0) = s(\U_0) = \U$.

One shows similarly, that $f(\U_j) = \U$ for $j=1,2,3$.  For instance $h(\U_3) = \{(x,y)\in\R^2:(-x,-y)\in \U_0\}$.  Hence $f(\U_3) = g(h(\U_3)) = g(\U_0) = \U$.
\end{proof}

Proposition \ref{prop:partition} suggests that $\U_j$, $j=0,1,2,3$ might be a Markov partition for the dynamics of $f$ on $\U$.  The presence of points of indeterminacy in $\partial\U$ makes this idea a bit tricky to verify, but it is nevertheless carried out in detail in \cite{han}, which gives a complete account of both the real and complex dynamics of $f$ on all of $\P^2$.  Our purpose here is to consider the dynamics of a family $f_t$, $t\in[0,1]$ of perturbations of $f$, so we turn now to these.

For each $t\in[0,1]$ we set $f_t:=T_t\circ f$ where $T_t$ is the linear map given in homogeneous coordinates by 
$$
T_t[x_1:x_2:x_3] = [x_1:x_2:x_3]-\frac{1-t}{3}(x_1+x_2+x_3)[1:1:1].
$$
Then $T_0$ is the central projection from $\p:=[1:1:1]$ to the line $L_0 := \{x_1+x_2+x_3=0\}$, and in adapted affine coordinates $T_t$, $t\neq 0$ is the scaling map $(x',y')\mapsto t^{-1}(x',y')$.  In particular, for $t\in (0,1)$, we have that $\U_t := T_t(\U) \supset\overline{\U}$.  Since $\U$ is totally invariant by $f$, it follows that 
\begin{equation}
\label{EQN_BACKWARD}
f_t^{-n}\left(\overline{\U}\right) \subset \U \text{ for all } n\in\N.
\end{equation}
As in \S \ref{SEC_PROOFS}, the indeterminacy set $\I := \I(f) = \I(f_t)$ is independent of the perturbation. 

Our goal in the remainder of this section is to establish the following two results about dynamics of $f_t$.

\begin{theorem}\label{THM_REAL_MEASURE}
For all $t \in (0,1)$, the following are true for $f_t$.
\begin{itemize}
 \item[(a)] $f_t$ is algebraically stable with dynamical degrees $\lambda_1=\lambda_2 = 4$.
 \item[(b)] There is no $f_t$-invariant foliation.   
 \item[(c)] The topological entropy of $f_t$ as a real (and complex) map is $\log 4$.
 \item[(d)] More precisely, $f_t$ admits an ergodic invariant measure $\mu$ of (maximal) entropy $\log 4$ with $\supp\mu\subset\U$ 
 \item[(e)] $\supp\mu$ is not contained in any algebraic curve.
\end{itemize}
\end{theorem}

\begin{theorem}\label{THM_TWO_MEASURES}
For all positive $t$ close enough to $0$, $f_t$ admits exactly two ergodic measures $\mu$ (the measure in Theorem \ref{THM_REAL_MEASURE}) and $\nu$ of maximal entropy $\log 4$.  Moreover,
\begin{itemize}
\item[(f)] $f_t$ is hyperbolic of repelling type on $\A := \supp\mu$.
\item[(g)] $f_t$ is hyperbolic of saddle type on $\Omega:= \supp\nu$.
\item[(h)] Periodic points are dense in $\A$ and $\Omega$
\item[(i)] Neither measure is supported on an algebraic curve.
\item[(j)] $\A$ is real whereas $\Omega$ is not.
\item[(k)] $f_t$ has three periodic points $\fp_1,\fp_2,$ and $\fp_3$ not in $\A\cup\Omega$, all fixed and attracting.  
\item[(l)] For any $z\in\P^2$, exactly one of the following occurs.
\begin{itemize}
\item $f_t^n(z)\in \A$ for all $n\geq 0$.
\item $f_t^n(z) \in \I$ for some $n\geq 0$.
\item $f_t^n(z) \to \Omega$ as $n\to\infty$.
\item $f_t^n(z)$ tends to one of the attracting point $\fp_j$ as $n\to\infty$.
\end{itemize}
\end{itemize}
\end{theorem}

\begin{remark} If $t\in (0,1)$ is not close to $0$, then we do not know whether $f_t$ admits measures (real or complex) of maximal entropy different from $\mu$.
\end{remark}

\subsection{Dynamical properties of $f_t$ that hold for any $t\in (0,1)$.}
Allowing for now that $t$ is any parameter in $(0,1)$, we will prove the assertions in Theorem \ref{THM_REAL_MEASURE} more or less in order.  

\begin{proof}[Proof of (a)] We already have $d(f_t) = d(f) = 4 = \lambda_2(f) = \lambda_2(f_t)$.
Recall that $\I = \I(f) \subset \partial{\U}$.  Hence by \eqref{EQN_BACKWARD}, the entire backward orbit of $\I$ is contained in $\U$.  On the other hand, $f_t(\E) = T_t(\I) \subset f_t(\partial{\U})$ does not meet $\U$.  Hence $f_t$ is algebraically stable, and it follows that $\lambda_1(f_t) = d(f_t) = 4$.  
\end{proof}

\begin{proof}[Proof of (b)]
Let $p=a_j$ be any point in $\I$.  It suffices to show that both criteria of
Theorem \ref{THM_CONDITIONS_FOR_NO_FOLIATION} are satisfied.  Since
$p\in\partial{\U}$ and $f_t(\E)$ lies outside $\overline{\U}$, it follows from
\eqref{EQN_BACKWARD} that the entire backward orbit of $p$ is contained in
$\U\setminus\E$.  In particular, $f_t$ is a finite map at each point of the
entire backward orbit of $p$.  Meanwhile, since the branch locus of $f_t$ is
also disjoint from $\U$, we infer that $f$ is regular at every point in the
backward orbit of $p$.  In particular, the backward orbit has infinitely many
distinct points.  So the second condition of Theorem
\ref{THM_CONDITIONS_FOR_NO_FOLIATION} holds.

For the first condition, note that $f_t(p) = T_t(Q)$.  In particular $f_t(p)\cap\overline{\R^2} = \partial\U_t$ is disjoint from $\overline{\U}$.  So from \eqref{EQN_BACKWARD} again, we infer that $f_t^n(p)\cap \I = \emptyset$ for all $n > 0$.  In particular $f_t^n(p)$ is irreducible for all $n > 0$.  
It remains to show that $\deg f_t^n(p) \to \infty$.    Notice that the diagonal line $L =\{[x:x:z]\}$ is invariant
under $f_t$.  Moreover, $L$ is not contained in $f_t^n(p)$ for $n~>~0$, since the latter is irreducible and symmetric in the homogeneous coordinates $x_1,x_2,x_3$.  Consider the
intersection between $f_t^n(p)$ for $n\geq 2$ and $L$.
Since $f_t^{n-1}(p)$ meets the exceptional line $A_3$, and $f_t(A_3) \in
\partial\U_t\cap L$, it follows that $f_t^n(p)$ contains orbit segment
$f_t(A_3),\dots,f_t^{n-1}(A_3)\in L$.  It follows from \eqref{EQN_BACKWARD}
that the points in this segment are distinct.  Hence $\deg f_t^n(p) \geq \#
f_t^n(p)\cap L \geq n-1$ for all $n\in\N$.

\end{proof}

\begin{proof}[Proof of (c)]  
Since $T_t(\U) \supset\overline{\U}$ and $I(f_t) \subset\partial\U$, it follows
from Proposition \ref{prop:Q} and Equation (\ref{EQN_BACKWARD})
that 
$$
\K := \bigcap_{n \geq 0} f_t^{-n}(\overline{\U})
$$
is a totally invariant subset of $\U$.  Proposition \ref{prop:partition} gives
that $f_t^{-1}(\overline{\U})$ is a disjoint union of four compact sets
$\U_j\cap f_t^{-1}(\overline{\U}) \subsetneq \U_j$, $j=0,1,2,3$, and that $f_t$ maps each of
these sets homeomorphically onto $\overline{\U}$.  Hence we can assign to each
$p\in \K$ its \emph{itinerary} $\iota(p) = (i_n)_{n\geq 0}$ where $i_n\in
\{0,1,2,3\}$ is chosen to satisfy $f_t^n(p) \in \U_{i_n}$.  Standard arguments
then tell us that the map $\iota:\K \to \Sigma_4$ from points to itineraries is
a continuous semiconjugacy onto the full 4-shift $\sigma:\Sigma_4\to \Sigma_4$
(given by $(i_n)\mapsto (i_{n+1})$).  It follows that $\log 4 \geq h_{top}(f_t)
\geq h_{top}(\sigma) = \log 4$. 
\end{proof}

\begin{proof}[Proof of (d)]
 The $4$-shift admits a unique measure of maximal entropy $\check \mu$, and the support of this measure is all of $\Sigma_4$.  Starting with any (not necessarily invariant) measure $\mu_0$ on $\K$ such that $\iota_*\mu_0 = \check\mu$, any weak limit $\mu$ of the sequence
$$
\frac{1}{N+1}\sum_{n=0}^N f_{t*}^n\mu_0
$$
will be an invariant measure satisfying $\iota_*\mu = \check\mu$.  Taking an extreme point of
all of the possible limits, the measure can be assumed to be ergodic.  The inequalities
$
h_{\check\mu}(\sigma) \leq h_\mu(f_t)\leq h_{top}(f_t)
$
imply that that $h_\mu(f_t) = \log 4$.
\end{proof}

\begin{proof}[Proof of (e)]
Here we reuse some of the arguments for Theorem \ref{THM_MU_NOT_1D}.  Suppose $\supp\mu$ is contained in an algebraic curve $C\subset\P^2$.  We may assume that each component of $C$ intersects $\supp \mu$ in a set of positive measure, containing infinitely many points.  Ergodicity of $\mu$ implies that the components of $C$ are permuted in a single cycle.  The arguments for Lemma \ref{LEM_SMOOTH_CURVE} imply (again) that $f_t^{-1}(C) = C$ and that no component of $C$ is critical for $f_t$.  Thus
\begin{align*}
f_t^* C = C + \sum n_i A_i
\end{align*}
with at least one coefficient, say e.g. $n_3$, positive.  Hence $f_t(A_3)\in C$, and by invariance of $C$ the entire forward orbit of $A_3$ is contained in $C$.  It is also, however, an infinite set contained in the forward invariant diagonal line $L = \{[x:x:z]\}$ that we considered in part (c).  Since the components of $C$ are cyclically permuted by $f_t$, each of them intersects $L$ in infinitely many distinct points.  We conclude that $C = L$.  A simple calculation shows that $L$ is not backward invariant, and this contradiction concludes the proof.
\end{proof}

\subsection{Dynamical properties of $f_t$ that hold for $t$ near $0$}

\noindent
It turns out that the dynamics of the limiting one dimensional map $f_0:L_0\to L_0$ can be fairly easily understood.  In particular, the points $[1:-1:0], [-1:0:1], [0:1:-1]$ where $L_0$ meets the $f$-invariant branch curves $X_1,X_2,X_3$ are necessarily fixed and attracting for $f_0$.  Let $z$ denote the (unique) affine coordinate on $L_0$ that identifies these three points with $0,1,\infty\in\P^1$.  One finds by direct computation that $f_0|_{L_0}$ becomes the rational function
$$
\label{EQN_1D_MAPPING_R}
r(z) = -{\frac {z \left( 5\,{z}^{3}-12\,{z}^{2}+6\,z-4 \right) }{4\,{z}^{3}-6
\,{z}^{2}+12\,z-5}}.
$$
Figure \ref{FIG_JULIA} shows a computer generated image of the basins of attraction the three fixed points of $r(z)$.

\begin{figure}
\includegraphics[scale=0.26]{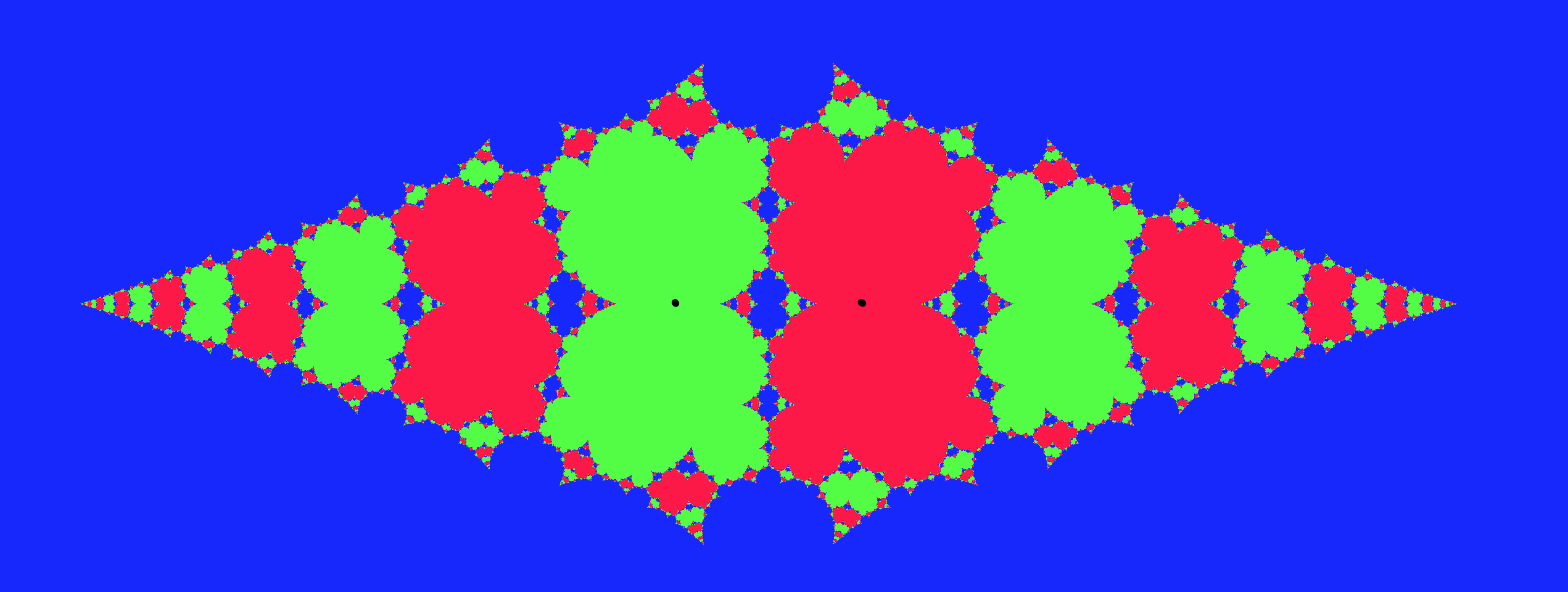}
\caption{Julia set for $f_{0|L_0} : L_0 \rightarrow L_0$, with the basins of attraction of $0$ in
green, $1$ in red, and $\infty$ in blue.  The whole figure corresponds
(approximately) to  $-3.5~\leq~{\rm re}(z)~\leq~4.5$ and $-1.5 \leq {\rm im}(z)
\leq 1.5$.\label{FIG_JULIA}}
\end{figure}

\begin{lemma}\label{LEM_J1_HYPERBOLIC} The Julia set $\J$ for $f_0|_{L_0}$ is hyperbolic and connected.  Moreover $L_0\setminus \J$ consists of the basins of the three attracting fixed points $[1:-1:0],[-1:0:1],[0:1:-1]\in L_0$.  Consequently $\J$ is not real, i.e. not contained in $L_0\cap\R\P^2$.
\end{lemma}

\begin{proof}
In addition to the three attracting fixed points: $0, 1$ and $\infty$, the above map $r(z)$ has six
critical points, which come in three conjugate pairs.  
By the Fatou-Julia Lemma, each of the attracting fixed points must have at least one
critical point in its immediate basin of attraction.  However, because of the real symmetry, each attracting fixed point actually has a conjugate pair of critical points in its immediate basin,
implying that $\J$ is hyperbolic.

To see that $\J$ is connected, we will check that each Fatou component of $r$
is homeomorphic to a disc.  By the threefold symmetry permuting $0$, $1$, and
$\infty$, it suffices to do this for each component in the basin of attraction
of $0$.  The immediate basin of $0$ is mapped to itself by $r$ under a ramified
cover with two simple critical points.  The Riemann-Hurwitz Theorem can be used to
rule out all possible degrees of this cover other than degree $3$, in which
case it implies that the immediate basin has Euler characteristic $1$.  Since
$r$ has no critical points outside the immediate basins of $0,1,$ and $\infty$,
each of the other components of the basin of attraction of $0$ is mapped
conformally onto the immediate basin by an iterate of $r$.  Thus, these
components are discs as well. 
\end{proof}

All conclusions of Theorem \ref{THM_TWO_MEASURES} now proceed directly from Lemma \ref{LEM_J1_HYPERBOLIC}, Theorem \ref{THM_REAL_MEASURE}, and the earlier Theorems \ref{thm:muexists} and \ref{thm:nuexists}.

\subsection*{Acknowledgments}
We thank Micha{\l} Misiurewicz for helpful discussions.  We also thank the referee for his or her very careful reading and for many good suggestions for improving this paper.
The first two authors were supported by NSF grant DMS-1066978 and the third author by NSF grant DMS-1348589.

\bibliographystyle{plain}

\end{document}